\documentclass[psamsfonts]{amsart}
\pdfoutput =1

\usepackage[english]{babel}
\usepackage[utf8x]{inputenc}
\usepackage[T1]{fontenc}
\usepackage[parfill]{parskip}
\usepackage{amsmath,amssymb}
\usepackage{mathtools}
\usepackage{amsfonts}
\usepackage[mathscr]{euscript}
\usepackage{array}
\usepackage{caption}
\usepackage{tikz}
\usepackage{tikz-cd}
\usetikzlibrary{arrows, matrix}
\usepackage{cite}
\usepackage{spectralsequences}

\usepackage[top=3cm,bottom=3cm,left=4cm,right=4cm,marginparwidth=1.75cm]{geometry}

\usepackage[all,arc]{xy}
\usepackage{enumerate}
\usepackage{mathrsfs}

\newtheorem{thm}{Theorem}[section]
\newtheorem{cor}[thm]{Corollary}
\newtheorem{proposition}[thm]{Proposition}
\newtheorem{lem}[thm]{Lemma}

\theoremstyle{definition}
\newtheorem{defn}[thm]{Definition}

\newtheorem{notn}[thm]{Notation}

\theoremstyle{remark}
\newtheorem{remark}[thm]{Remark}

\newtheorem{example}[thm]{Example}

\makeatletter
\let\c@equation\c@thm
\numberwithin{equation}{section}
\makeatother

\DeclareMathOperator{\Ima}{Im}

\DeclareMathOperator{\Gr}{gr}
\DeclareMathOperator{\sgn}{sgn}

\DeclareMathOperator{\coker}{coker}
\newcommand{\twisty}{\mathbin{\widetilde{\otimes}}}
\newcommand{\squarebin}{\mathbin{\square}}
\newcommand{\FDGM}{\textrm{FDGM}}
\newcommand{\SSEQ}{\textrm{SSEQ}}

\usepackage{geometry}

\title{Comultiplication in the Serre Spectral Sequence}
\author{David Chan}

\begin{document}
	\maketitle
	\begin{abstract}
		We show the homological Serre spectral sequence with coefficients in a field is a spectral sequence of coalgebras.  We also identify the comultiplication on the $E^2$ page of the spectral sequence as being induced by the usual comultiplication in homology.  At the end, we provide some example computations highlighting the use the co-Leibniz rule.
	\end{abstract}
	\section{Introduction}	 
		An overarching theme in algebraic topology is that more algebraic structure is always better.  A basic example appears in an introductory course on algebraic topology.  Let $X=S^2\vee S^4$ be the wedge of the $2$ and $4$-spheres and let $Y=\mathbb{C}P^2$ be the complex projective plane.  An easy calculation using CW structures tells us the integral homology groups of $X$ and $Y$ are isomorphic: both have copies of $\mathbb{Z}$ in dimensions $0$, $2$ and $4$ and are $0$ in other dimensions.  A little more work shows that the cohomology \emph{rings} of $X$ and $Y$ are not isomorphic, as the cohomology of $Y$ admits non-trivial products in dimension $2$.  The upshot is $X$ and $Y$ cannot be homotopy equivalent, despite having isomorphic homology groups.
		
		As this example suggests, cohomology lends itself to certain arguments that are missed by considering only the homology of spaces.  Indeed, in general there is no corresponding structure on the homology of spaces that is equivalent to the ring structure in cohomology.  However, when coefficients for homology are taken to be a field $k$, the homology of a space admits the structure of a $k$-coalgebra.  While this structure is completely dual to the $k$-algebra structure on cohomology, it is often passed over in favor of the more familiar algebra.  The purpose of this paper is to investigate a place in the theory where the coalgebra structure on homology of spaces has been largely ignored. 
		
		The cohomological Serre spectral sequence associated to a fibration is an important tool for both computation and theory in algebraic topology.  Among other things, it was used for groundbreaking computations in the homotopy groups of spheres as well as the homology of Eilenberg--MacLane spaces.  One of the primary features of this spectral sequence that makes computation achievable is that this a spectral sequence of algebras.  The multiplications in the spectral sequence, combined with the Leibniz rule, give an important avenue of attack for resolving differentials.  Those interested in applications of this, and other spectral sequences, are recommended to \cite{Mc}.
	
		Dual to the cohomological Serre spectral sequence there is also a spectral sequence for relating the homology groups of spaces in a fibration.  However, just as the homology of a space is not a ring, the homological Serre spectral sequence is not a spectral sequence of algebras.  As such, the problem of resolving differentials in this spectral sequence presents hurdles not found in the cohomological spectral sequence.  More precisely, the homological spectral sequence faces the same challenges but without the same toolkit for solving them.  In many cases, this is not too much of an issue; various duality arguments often allow one to get away with simply computing cohomology. Nevertheless, one often wants to compute the homology of spaces directly and so it is desirable to have more algebraic structure to exploit in the homological Serre spectral sequence.
		
		A reasonable question to ask is whether the homological Serre spectral sequence with coefficients in a field should be a spectral sequence of coalgebras.    A result of Quillen \cite[Section 6]{Qui:Rat} shows this is often true in the rational case.
		
		\begin{thm}[Quillen]
			Let $F\to E\xrightarrow{p} B$ be a Serre fibration with $B$ simply connected and $p$ inducing a surjection $p_*\colon\pi_2(E)\to \pi_2(B)$.  Then the homological Serre spectral sequence with coefficients in $\mathbb{Q}$ is a spectral sequence of $\mathbb{Q}$-coalgebras.  
		\end{thm}
		
		The main result of this paper is the following generalization of Quillen's result.
		\begin{thm}
			Let $k$ be any field and $F\to E\to B$ be a Serre fibration with $B$ simply connected.  Then the homological Serre spectral sequence with coefficients in $k$ is a spectral sequence of $k$-coalgebras. The comultiplication on the $E^2$ page can be identified with the usual comultiplication of the tensor product of coalgebras $H_{*}(B;k)\otimes H_*(F;k)$.
		\end{thm}
	
		Our proof is substantially different from Quillen's.  Indeed, Quillen's result is more of a corollary of his deep theorem regarding the equivalence of the rational homotopy category and the homotopy category of 2-reduced rational coalgebras.  Since there is no analogue of Quillen's theorem for fields other than $\mathbb{Q}$, his proof has no hope of extending beyond the rational case.
	
		The rest of the paper is organized as follows.  In section \ref{section: spectral sequences of coalgebras} we recall the definition of a spectral sequence of coalgebras.  We also discuss the notion of convergence of spectral sequences as coalgebras.  In section \ref{section: fdgms} we review filtered differential graded modules and how they give rise to spectral sequences.  In section \ref{section: fdgcs} it is shown that a filtered differential graded coalgebra over a field gives rise to a spectral sequences of coalgebras. This fact seems to be well-known to experts although the author does not know a reference.  The details of the proof are not used elsewhere in the paper.
		
		In sections \ref{section: EZ} and \ref{section: homology coalgebra}, we review some necessary facts about the Eilenberg--Zilber map and the homology coalgebra of a topological space.  These sections provide necessary notation and groundwork for the proof of the main theorem which appears in section \ref{section: SSS}.  The proof that the homological Serre spectral sequence is a spectral sequence of coalgebras is relatively easy given the result of section \ref{section: fdgcs}.  The bulk of the work in proving the main result is in identifying the comultiplication in the $E_2$ page of the spectral sequence with the usual comulitplication on homology from section \ref{section: homology coalgebra}.  The paper concludes with section \ref{section: examples} with some example computations that highlight the use of the co-Leibniz rule and subtleties in our notion of convergence as coalgebras.
		
		\section{Spectral Sequences of Coalgebras} \label{section: spectral sequences of coalgebras}
		In this section we establish notation for spectral sequences and discuss the less well-known notion of a spectral sequence of coalgebras.  Fix a commutative ring $R$.
		\begin{defn}
			A  \emph{spectral sequence} over $R$ consists of the following data:
			\begin{enumerate}
				\item Bigraded $R$-modules $E^r_{*,*}$ for all $r\in \mathbb{Z}_{\geq0}$.  The module $E^r_{*,*}$ is called the $r$-th page of the spectral sequence. 
				\item Homomorphisms $d^r\colon E^r_{p,q}\to E^r_{p-r,q+r-1}$ with $d^r\circ d^r=0$ when this makes sense.
				\item $E^{r+1}_{p,q}$ is isomorphic to the homology of $(E^r_{*,*},d^r)$ at $(p,q)$.
			\end{enumerate}
		The category of spectral sequences over $R$ is denoted $\SSEQ_R$.
		\end{defn}

		We will denote a spectral sequence by $(E,d)$, meaning that the $r$-th page will be $E^r_{*,*}$ and the differential on this page is $d^r$.  We note that each page $E^r$ is composed of countably many chain complexes with differentials given by the $d^r$.  The homology of the $E^r$ is just the homology of these chain complexes.  
		\begin{defn}
			Let $(E,d)$ and $(\widetilde{E},\widetilde{d})$ be two spectral sequences.  A \emph{morphism of spectral sequences} $\phi\colon E\to \widetilde{E}$ consists of $R$-linear maps $\phi^r\colon E^r_{*,*}\to \widetilde{E}^r_{*,*}$ for all $r$ such that:
			\begin{enumerate}
				\item $\phi^r\circ d^r = \widetilde{d}^r\circ \phi^r$ whenever this makes sense.
				\item After identifying $E^{r+1}\cong H(E^r)$ and $\widetilde{E}^{r+1}\cong H(\widetilde{E}^r)$ we have $\phi^{r+1}=H(\phi^r)$.
			\end{enumerate}
		\end{defn}
		We say that a morphism $\phi\colon E\to \widetilde{E}$ is an isomorphism if the maps $\phi^r$ are all isomorphisms.  As it turns out, it is sufficient to see that $\phi^0$ is an isomorphism.  Indeed, we have required that $\phi^0$ is a chain map of all the complexes comprising $E^0$. A chain map that is an isomorphism in all degrees must be a quasi-isomorphism. Thus the induced map $\phi^1$ is an isomorphism.  Working inductively, we see that $\phi^2$ is an isomorphism and so on.
		\begin{lem}\label{spec_iso}
			Let $\phi\colon E\to \widetilde{E}$ be a morphism of spectral sequences.  Then $\phi$ is an isomorphism of spectral sequences if and only if $\phi^0$ is an isomorphism of bigraded complexes.
		\end{lem}

		We now limit the scope of our discussion.  A first quadrant spectral sequence is a spectral sequence $(E,d)$ such that $E^{r}_{p,q}=0$ if either $p$ or $q$ is negative.  Restricting to first quadrant spectral sequences greatly simplifies the theory, due in large part to the following pleasant consequence.
		\begin{proposition}
			Let $(E,d)$ be a first quadrant spectral sequence.  For all $p$ and $q$ there exists an $r$ such that $E^s_{p,q}\cong E^{r}_{p,q}$ for all $s\geq r$.
		\end{proposition}
		\begin{proof}
			Let $r=\max(p+1,q+2)$.  Assume $s>r$, the module $E^{s}_{p,q}$ is isomorphic to the homology of the $(s-1)$-page at bidegree $(p,q)$.  By definition, this is the quotient $$\ker(d^{s-1}\colon E^{s-1}_{p,q}\to E^{s-1}_{p-s+1,q+s-2})/\Ima(d^{s-1}\colon E^{s-1}_{p+s-1,q-s+2}\to E^{s-1}_{p,q})$$
			Our assumption on $r$ implies that $p-s+1<0$ and $q-s+2<0$. Thus $E^{s-1}_{p-s+1,q+s-2} = E^{s-1}_{p+s-1,q-s+2} =0$ and the result follows.
		\end{proof}
		
		The proposition allows us to talk about the so-called infinity page of first quadrant spectral sequences, 
		\[
			E^r_{p,q} = E^{r+1}_{p,q}=\dots =\colon E^{\infty}_{p,q},  
		\]
		
		where $r$ is sufficiently large.  
		
		A major focus of this paper is on spectral sequences with the additional structure that every page is a coalgebra.  Recall an $R$ coalgebra is an $R$ module $C$, together with an $R$-linear maps $\triangle\colon C\to C\otimes C$ and $\epsilon \colon C\to R$ called the comultiplication and counit respectively.  The maps $\triangle$ and $\epsilon$ are subject to some associativity and compatibility conditions; we refer the reader to \cite{mayponto} for a full discussion. 
		\begin{defn}
			A spectral sequence $(E,d)$ over a field $k$ is a \emph{spectral sequence of coalgebras} over $k$ if:
			\begin{enumerate}
				\item Each page $E^r_{*,*}$ is a a bigraded coalgebra with comultiplication $\triangle^r$. 
				\item The comultiplications satisfy the co-Leibniz rule: $\triangle^r\circ d^r = (d^r\otimes 1 \pm 1\otimes d^r)\circ \triangle^r$ where the sign follows the usual Koszul sign convention.
			\end{enumerate}  
		\end{defn}
	
		\begin{remark}
			The category of spectral sequences over a ring can be given the structure of a symmetric monoidal category.  In this context, it turns out the definition above is equivalent to the notion of a comonoid in this category.  We will return to this perspective in Section \ref{section: fdgcs}.
		\end{remark}
		When $(E,d)$ is a first quadrant spectral sequnece of coalgebras the infinity page is naturally a bigraded coalgebra and we will denote the comultiplication by $\triangle^{\infty}$.
		
		We end this section with a discussion of convergence of first quadrant spectral sequences. There is subtlety surrounding the notion of convergence in the context of spectral sequences of coalgebras and we will need to be careful. 
		\begin{defn}
			Let $M$ be an $R$-module.  By a \emph{filtration} of $M$ we mean an increasing sequence of submodules 
			\[
				\dots \subset F^0\subset F^1\subset \dots\subset M
			\]
			If $M$ is a graded $R$-module we will denote the intersection $M_m\cap F^n$ by $F^n_m$. 
		\end{defn}

		\begin{defn}
			Let $M$ be a filtered graded $R$-module and $(E,d)$ a first quadrant spectral sequence.  We say that $E$ \emph{converges} to $M$, written $E\Rightarrow M$, if for all $p$ and $q$ there are isomorphisms:
			\begin{equation} \label{filtration quotients}
				E^{\infty}_{p,q}\cong F^{p}_{p+q}/F^{p-1}_{p+q}
			\end{equation}
		\end{defn}
		
		  A spectral sequence converges to a filtered graded module if the infinity page of the sequence tells you what the module is up to the extension problem induced by the filtration.  Of course, the extension problem is generally difficult to solve.  However, in the case when $R$ is a field it is trivial and we have:
		\begin{lem}\label{first quadrant convergence}
			Let $M$ be a filtered graded vector space over a field $k$ and let $(E,d)$ be a first quadrant spectral sequence over $k$ that converges to $M$.  Then we have isomorphisms:
			\[
				M_n\cong \bigoplus\limits_{p+q=n}E^{\infty}_{p,q}
			\]
		\end{lem}  
		When $M$ is a filtered vector space we denote the associated graded module by
		\begin{equation}\label{associated graded}
			\Gr(M) = \bigoplus\limits_{i\in \mathbb{Z}}F^{i}/F^{i-1}
		\end{equation}
		
		Another way to state Lemma \ref{first quadrant convergence} is that there are isomorphisms:
		\begin{equation}\label{gr_iso}
			\Gr(M)\cong \bigoplus\limits_{n\geq 0}\bigoplus\limits_{p+q=n}E^{\infty}_{p,q} 
		\end{equation}
		We use this isomorphism to discuss the notion of convergence of spectral sequences of coalgebras.  Suppose $C$ is a filtered graded coalgebra with comultiplication $\triangle$.  We will assume $\triangle$ respects the filtration in the sense that $\triangle(F^m)\subset F^m\otimes F^m$ for all $m$.  One can check there is an induced comultiplication $\Gr(\triangle)$ on $\Gr(C)$.  Together with the obvious map $\Gr(\epsilon)\colon \Gr(C)\to R$, we see that $\Gr(C)$ can be given the structure of a bigraded coalgebra.
		 
		\begin{defn}\label{convergence as coalgebras}
			In the notation of the last paragraph, we say that $(E,d)$ converges to $(C,\triangle)$ \emph{as a coalgebra} if $E\Rightarrow C$ and the isomorphism \eqref{gr_iso} is an isomorphism of bigraded coalgebras.
		\end{defn}
		\begin{remark}
			It is possible to construct non-isomorphic coalgebras $(C,\triangle_C)$ and $(D,\triangle_D)$ that are both the limit of a single spectral sequence of coalgebras $(E,d)$.  To see this, it is enough to construct a single filtered vector space $C=D$ with two different coalgebra structures $\triangle_1$ and $\triangle_2$ that induce the same comultiplcation $\Gr(\triangle_1) =\Gr(\triangle_2)$.  In particular, even if a spectral sequence converges to $(C,\triangle)$ as a coalgebra we cannot always recover the comultiplcation $\triangle$ even if we have full knowledge of $E^{\infty}$.  
		\end{remark} 

	\section{Filtered Differential Graded Modules}\label{section: fdgms}
		In this section we establish notation for filtered differential graded modules and  the associated spectral sequences.  Throughout, we will fix a unital, associative, and commutative ring $R$ and all modules will be assumed to be $R$-modules unless otherwise stated.
		
		By a filtered differential graded module $M$ we mean a homologically graded differential module $M$ with increasing filtration.  That is, $M$ is a graded module with boundary map $\partial\colon M\to M$ that lowers degree by $1$, and has a filtration 
		\[
			\dots\subset F^{p-1}\subset F^p\subset F^{p+1}\subset\dots \subset M
		\]
		We require that our boundary map $\partial$ respects this filtration in the sense that $\partial(F^n)\subset F^n$.  We will not assign different names to the boundary maps for each $F^p$, denoting all of them by $\partial$.  Similarly we will use $\iota\colon F^p\to F^{p+k}$ for all inclusions of the $F^i$.  A filtered differential graded module will be denoted by $(M,\partial,\{F^p\})$, and we write $F^p_q$ for the intersection $F^p\cap M_q$.

		A morphism of filtered differential graded modules is a chain map that preserves the filtrations.  More precisely, given two filtered differential graded modules $(M,\partial_M,\{F^p\})$ and $(N,\partial_N,\{G^p\})$, a morphism $\varphi$ between the two is an $R$-module homomorphism $\varphi\colon M\to N$ such that 
		\begin{enumerate}
			\item $\varphi\circ \partial_M = \partial_N\circ \varphi$
			\item $\varphi(F^p)\subset G^p$ for all $p$.
		\end{enumerate}
		The category of all filtered differential graded modules over $R$ and morphisms between them will be denoted by $\FDGM_R$.  The construction of a spectral sequence from a filtered differential graded module will be a functor $\FDGM_R\to \SSEQ_R$.

		Fix an $(M,\partial,\{F^p\})\in \FDGM_R$ and define submodules
		\begin{align*}
			Z^r_{p,q} & = F^p_{p+q}\cap \partial^{-1}(F^{p-r}_{p+q-1}) \\
			B^r_{p,q} & = F^p_{p+q}\cap \partial(F^{p+r}_{p+q+1}) 
		\end{align*}
		The elements of $Z^r_{p,q}$ are called the $r$ almost cycles and the elements of $B^r_{p,q}$ are the $r$ almost boundaries. The elements in $Z^r_{p,q}$ are the elements in $F^p_{p+q}$ whose boundaries descend $r$ levels down in the filtration.  Similarly, the elements in $B^r_{p,q}$ are the elements in $F^p_{p+q}$ which are boundaries of elements $r$ levels up in the filtration.  
		
		\begin{remark} \label{B is image of Z}
			Following the definitions we see that $\partial(Z^r_{p,q})=B^r_{p-r,q+r-1}$.
		\end{remark}
		
		For fixed $p$ and $q$, the $Z^r_{p,q}$ and $B^r_{p,q}$ fit into a tower of submodules.
		\begin{lem} \label{inclusion lemma}
			For all $r,s>1$ we have the following chain of inclusions:
			\[
				B^1_{p,q}\subset \dots B^r_{p,q}\subset\dots\subset Z^s_{p,q}\subset\dots \subset Z^1_{p,q}
			\]
		\end{lem}
		
		This lemma justifies the following definition of our spectral sequence terms.
		\begin{defn}
		Given $M\in \FDGM_R$, the $r$-th pages of the associated spectral sequence are given by:
		\[
			E^r_{p,q}=(Z^r_{p,q}+F^{p-1}_{p+q}) / (B^{r-1}_{p,q}+F^{p-1}_{p+q})
		\]
		\end{defn}
		\begin{remark}
			This definition for the terms of our spectral sequence may not be as familiar to the reader as some other presentations.  This choice of spectral sequence terms comes from Weibel's book \cite{Weibel}.  While different choices, such as those in \cite{Mc} or \cite{CE}, give equivalent information we find this form best suits our purposes in Section \ref{section: fdgcs}.
		\end{remark}
		
		By Remark \ref{B is image of Z} and Lemma \ref{inclusion lemma} we have $\partial\colon Z^r_{p,q}\to Z^r_{p-r,q+r-1}$.  Since we have 
		\[	
			\partial(B^{r-1}_{p,q}+F^{p-1}_{p+q}) \subset 0+F^{p-1}_{p+q-1},
		\]
		we get an induced map $d^r_{p,q}:E^r_{p,q}\to E^r_{p-r,q+r-1}$ that fits into the following commutative diagram:
		\begin{equation} \label{sequence square}
		\begin{tikzcd}
			{Z^r_{p,q}} \arrow[r, "\partial"] \arrow[d, "\pi^r_{p,q}"'] & {Z^r_{p-r,q+r-1}} \arrow[d, "\pi^r_{p-r,q+r-1}"] \\
			{E^r_{p,q}} \arrow[r, "d^r_{p,q}"]                       & {E^r_{p-r,q+r-1}}                 
		\end{tikzcd}
		\end{equation}
		where the $\pi'$s are the canonical quotient maps.  The subscripts $p,q$ may be omitted when no confusion is possible. The following proposition is standard and a proof can be found in \cite{Weibel}.

		\begin{proposition} \label{functor}
			The collection $(E,d)$ is a spectral sequence over $R$.  Moreover, the formation of this spectral sequence from a filtered differential graded module specifies a functor $\textbf{E}\colon \FDGM_R\to \SSEQ_R$.
		\end{proposition}
	
		In nice cases, this spectral sequence converges to the homology of our differential graded module.
		\begin{defn}\label{first quadrant filtration}
			Let $\{F^p \}$ be a filtration of of a graded module $M$.  We say $\{F^p \}$ is a \emph{first quadrant filtration} if in each grading $n$ the filtration has the form:
			\[
				0 = \dots = F^{-1}_n\subset F^{0}_n\subset\dots\subset F^n_n = F^{n+1}_n = \dots=M_n
			\] 
		\end{defn}
		As one might expect, the terminology here reflects the fact that a filtered differential graded module with a first quadrant filtration leads to a first quadrant spectral sequence.  This is essentially obvious, as one can check that:
		\[
			E^0_{p,q} = F^{p}_{p+q}/F^{p-1}_{p+q}
		\]
		implying that for first quadrant filtrations $E^0_{p,q}=0$ for all $p$ or $q$ less than $0$. 
		\begin{proposition}
			Let $(M,\partial,\{F^p \})$ be a filtered differential graded module with a first quadrant filtration.  Then the homology of $M$ is a filtered graded module and the spectral sequence $\textbf{E}(M)$ converges to $H_*(M)$.
		\end{proposition}
		
	\section{Filtered Differential Graded Coalgebras}\label{section: fdgcs}
	
		Let $R=k$ be a field.  In this section we consider conditions under which the spectral sequence associated to an object $(C,\partial,\{F^p\})$ of $\FDGM_k$ admits a coalgebra structure.   
		\par 
		Recall that the grading on $C$ allows us define a grading on $C\otimes C$ by 
		\[
			(C\otimes C)_n = \bigoplus\limits_{a+b=n}C_a\otimes C_b
		\]
		We will also extend the differential structure of $C$ to one on $C\otimes C$ with differential $\delta$ of $C\otimes C$ given by 
		\[
			\delta_n = \bigoplus\limits_{a+b=n} \partial\otimes 1 + (-1)^a1\otimes \partial
		\]
		Finally, we extend the filtration on $C$ to a filtration $\{G^p\}$ of $C\otimes C$, called the filtration induced by $F$, by defining 
		\[
			G^p=\sum\limits_{c+d=p} F^c\otimes F^d
		\]
		 
		It is easily checked that $(C\otimes C,\delta,\{G^p\})$ is also a filtered differential graded module.  When the map $\triangle\colon C\to C\otimes C$ is a morphism in $\FDGM_k$, we will call $C$ a filtered differential graded coalgebra. 
		
		We will work to identify the spectral sequence $(E,e)=\textbf{E}(C\otimes C,\delta,\{G^p\})$ in terms of the spectral sequence $(D,d)=\textbf{E}(C,\partial,\{F^p\})$. To be precise, we define a new spectral sequence $(D \otimes D,d\otimes d)$ with $r$-th page 
		\[
			(D\otimes D)^r_{p,q} = \bigoplus\limits_{a+b=p}\bigoplus\limits_{c+d=q} D^r_{a,c}\otimes D^r_{b,d}
		\] 
		and differential  $d\otimes d$ given on each level of the direct sum by $(d^r\otimes 1 + (-1)^{a+c}1\otimes d^r)$. That each page is given as the homology of the previous follows from the Kunneth isomorphism.  We will show that $(E^r,e^r )$ is isomorphic to $(D\otimes D,d\otimes d)$ with these differentials. 
		 
		To fix some notation, we will use $Y^r_{p,q}$ and $A^r_{p,q}$ for the $r$ almost cycles and boundaries of $C$ and $Z^r_{p,q}$ and $B^r_{p,q}$ for the $r$-almost cycles and boundaries of $C\otimes C$.
		\begin{lem}\label{identify Z}
			In the notation above:
			\[
				Z^r_{p,q} +G^{p-1} = \bigoplus\limits_{c+d=p+q}\sum\limits_{a+b=p} Y^r_{a,c-a}\otimes Y^r_{b,d-b}+G^{p-1}.
			\]
		\end{lem}
		\begin{proof}
			For the forward inclusion, let $x\in Z^r_{p,q}\subset G_{p+q}^p$. We may write $x=\sum y^a_{c,k}\otimes z^b_{d,k}$ for $y^a_{c,k}\in F^a_c$ and  $z^b_{d,k}\in F^b_{d}$.  The sum runs over $c+d=p+q$, $a+b=p$ and $k$.  The $k$'s here reflect the fact that the component of $x$ in any particular $F^a_{c}\otimes F^b_{d}$ may not be a simple tensor.  Since this will have no impact on the rest of the argument they will be omitted to clear up notation.  If any of the $y^a_{c}\in F^l_c$ for $l<a$, then $y^a_{c}\otimes z^b_{d}\in G^{p-1}$ and so $\delta(y^a_{c}\otimes z^b_{d})\in G^{p-1}$ is an element of the module on the right.  We can therefore ignore all such terms and similarly when $z^b_{d}\in F^l_d$ for $l<b$.
			\par
			Since $x\in Z^r_{p,q}$ we have $\delta(x)\in G^{p-r}_{p+q-1}$, or 
			
			\[
				\delta(x) = \sum (\partial y^a_{c}\otimes z^b_{d}) +(-1)^c (y^a_{c}\otimes \partial z^b_{d})\in \bigoplus\limits_{i+j=p+q-1}\sum\limits_{s+t=p-r} F^s_{i}\otimes F^t_{j}
			\]
			Since $z^b_d\notin F^l_j$ for $l<b$, we must have $\partial(y^a_c)\in F^{p-r-b}_{c-1}=F^{a-r}_{c-1}$.  Similarly $\partial(z^b_{d})\in F_{d-1}^{b-r}$.  By definition $y^a_c\in Y^r_{a,c-a}$ and $z^b_d\in Y^r_{b,d-b}$.
			\par
			For the reverse inclusion it suffices to show that
			\[
				Y^{r}_{a,c-a}\otimes Y^{r}_{b,d-b}\subset Z^r_{p,q}
			\]
			Given $x\otimes y \in Y^{r}_{a,c-a}\otimes Y^{r}_{b,d-b}$, we need to show that $\delta(x\otimes y)\in G^{p-r}_{p+q-1}$.  By definition we have:
			\[
				\delta(x\otimes y) = \partial(x)\otimes y + (-1)^c x\otimes\partial(y) 
			\]
			As $x\in Y^r_{a,c-a}$ we see $\partial(x)\in F^{a-r}_{c-1}$ and similarly $\partial(y)\in F^{b-r}_{d-1}$.  Thus $\partial(x)\otimes y$ and $x\otimes \partial(y)$ are elements of $G^{a+b-r}_{c+d-1} = G^{p-r}_{p+q-1}$ and we are done.
		\end{proof}
		
		\begin{proposition} \label{tensor ID}
			There is a spectral sequence isomorphism $\varphi\colon (E,e)\to (D\otimes D,d\otimes d)$.  That is, for all $p,q$, and $r$ there are isomorphisms
			\[
				\varphi^r_{p,q}:E^r_{p,q}   \to \bigoplus\limits_{c+d=q} \bigoplus\limits_{a+b=p} D^r_{a,c}\otimes D^r_{b,d}
			\]
			that satisfy
			\[
				\varphi^r\circ e^r = (d^r\otimes 1 + (-1)^{a+c}1\otimes d^r) \circ \varphi^r
			\]
		\end{proposition}
		\begin{proof}
			Using Lemma \ref{identify Z}, we can write 
			\[
				E^r_{p,q} = \frac{\bigoplus\limits_{c+d=p+q}\sum\limits_{a+b=p}Y^r_{a,c-a}\otimes Y^r_{b,d-b}+G^{p-1}}{\delta(Z^r_{p+r-1,q-r+2})+G^{p-1}} 
			\]
			We define $\varphi^r_{p,q}$ on generators of the numerator by $\varphi^r_{p,q}([x\otimes y]) \mapsto [x]\otimes [y]\in D^r_{a,c}\otimes D^r_{b,d}$ for all generators $x\otimes y \in Y^r_{a,c}\otimes Y^r_{b,d}$ (the generators in $G^{p-1}$ must go to zero).  This is a well defined function since the intersection of two of the subspaces in the sum is contained in $G^{p-1}$.  Writing $F=F^{p-1}_{p+q}$, we see this is a homomorphism by identifying $D^r_{a,c}\otimes D^r_{b,d}$ with
			\begin{equation} \label{DtensorD}
				\frac{Y^r_{a,c}+F}{A^{r-1}_{a,c}+F}\otimes \frac{Y^r_{b,d}+F}{A^{r-1}_{b,d}+F} \cong 
				\frac{(Y^r_{a,c}+F)\otimes (Y^r_{b,d}+F)} {(Y^r_{a,c}+F)\otimes(A^{r-1}_{b,d}+F^) + (A^{r-1}_{a,c}+F)\otimes (Y^r_{b,d}+F)}
			\end{equation}
			and seeing the quotiented part of $E^r_{p,q}$ is carried to 0.  The second claim follows from:
			\begin{align*}
				\varphi^r_{p,q}(\delta([x\otimes y]))  & = \varphi^r_{p,q}([\partial(x)\otimes y+(-1)^{a+c}x\otimes \partial(y)]) \\
				& = [\partial(x)]\otimes[y]+ (-1)^{a+c}[x]\otimes [\partial(y)]\\
				& = (d^r\otimes 1 + (-1)^{a+c}1\otimes d^r)([x]\otimes [y]) \\
				& = ((d^r\otimes 1 + (-1)^{a+c}1\otimes d^r)\circ \varphi^r_{p,q})([x\otimes y])
			\end{align*}

			All that remains is to show this map is an isomorphism. By  Lemma \ref{spec_iso}, it is sufficient to show $\varphi^0$ is an isomorphism of the zero-th pages. Chasing the definitions a bit, $Z^0_{p,q}=G^p$ and $B^{-1}_{p,q} = \delta(Z^{-1}_{p-1,q})\subset G^{p-1}$ and we just get $E^0_{p,q} = G^{p}_{p+q}/G^{p-1}_{p+q}$.  The map $\varphi^0$ is exactly the following:
			\begin{align*}
				G^p/G^{p-1} & = \sum\limits_{a+b=p} F^a\otimes F^b / \sum\limits_{c+d=p-1} F^c\otimes F^d \\
				& \cong \sum\limits_{a+b=p} (F^a\otimes F^b) / (F^a\otimes F^{b-1} + F^{a-1}\otimes F^b)\\
				& \cong \sum\limits_{a+b=p} F^a/F^{a-1} \otimes F^b/F^{b-1}
			\end{align*}
		\end{proof}
		We reinterpret Proposition \ref{tensor ID} in more categorical language. The reader unfamiliar with monoidal categories and strong monoidal functors should consult \cite[Section XI]{MacLane}.
		
		We may rewrite the isomorphism of Proposition \ref{tensor ID} as
		\begin{equation}\label{strong_monoidal}
			\textbf{E}(C\otimes C)\cong \textbf{E}(C)\otimes\textbf{E}(C)
		\end{equation}
		One can check that this identity is sufficiently natural, thus the functor $\textbf{E}\colon \FDGM_k\to \SSEQ_k$ is a strong monoidal functor of symmetric monoidal categories.  It follows formally that $\textbf{E}$ takes filtered differential graded coalgebras to spectral sequences of coalgebras as both are the correct notion of comonoids in their respective categories.  
		\begin{thm}\label{fdgcs gives spectral sequence of coalgebras}
			Suppose $C$ is filtered differential graded coalgebra.  Then $\textbf{E}(C)$ is a spectral sequence of coalgebras.  If the filtration of $C$ is a first quadrant filtration then $\textbf{E}(C)$ converges to $H_*(C)$ as a coalgebra.
		\end{thm} 
		\begin{proof}
			All there is left to prove is that the convergence is convergence of coalgebras in the sense of Definition \ref{convergence as coalgebras}.  The homology of $C$ is endowed with a comultiplication via the composition:
			\[
				H_*(C)\xrightarrow{\triangle} H_*(C\otimes C)\cong H_*(C)\otimes H_*(C)
			\]
			where the second isomorphism is the Kunneth isomorphism.  The fact that comultiplication in $E^{\infty}$ and in $H_*(C)$ are both induced by $\triangle$ implies the necessary compatibility condition of Definition \ref{convergence as coalgebras}.
		\end{proof}
  	\section{The Eilenberg--Zilber Map} \label{section: EZ}
  	
	  	In this short section we set up notation for $(q,p)$ shuffles and the Eilenberg--Zilber theorem.  Of particular importance is the notion of a simple $(q,p)$ shuffle, defined below.  This special class of shuffles will feature prominently in the proof of the main theorem. We assume the reader is familiar with the language of simplicial sets, a good reference is \cite{Goerss-Jardine}.
	  	
	  	We denote the simplicial category by $\mathbf{\Delta}$.  For non-negative integers $n$, we will let $\underline{n}$ denote the ordered set $\{0,1,\dots, n\}$, considered as an object in $\mathbf{\Delta}$.

		\begin{defn}
			A \emph{$(q,p)$ shuffle} is a pair $(\mu,\sigma)$ of morphisms $\mu\colon\underline{p+q}\to \underline{q}$ and $\sigma\colon\underline{p+q}\to \underline{p}$ in $\mathbf{\Delta}$ such that:
			\begin{enumerate}
				\item Both $\sigma$ and $\mu$ are surjective.
				\item For all $j$ so that $\mu(j)=\mu(j+1)$ we have $\sigma(j+1)=\sigma(j)+1$ and vice versa.
			\end{enumerate}
		\end{defn}
		 To get a feel for the second condition, consider the following $(5,3)$ shuffle:
		\begin{equation}\label{shuffle 1}
		(\mu,\sigma)=([0,0,1,2,3,3,3,4,5],[0,1,1,1,1,2,3,3,3])
		\end{equation}
		The second condition is saying that at every pair of indices $(j,j+1)$ such that $\mu$ does not increase, that is $\mu(j)=\mu(j+1)$, we must have $\sigma$ increase.  The main point is that the places where $\mu$ and $\sigma$ are degenerate complement one another.  As a result, $\mu$ is determined by $\sigma$ and so it should come as no surprise later on that many of our results involving $(q,p)$ shuffles depend only on $\sigma$.
		\begin{remark}
			The name $(q,p)$ shuffles comes from the observation that a $(q,p)$ shuffle is equivalent to taking two ordered decks of cards, of size $q$ and $p$ respectively, and shuffling them together so that the order of each deck is not changed.  That is, if card $A$ is below card $B$ in the deck of size $q$ then card $A$ should still be below card $B$ after shuffling.  In this way one can view $(q,p)$ shuffles as permutations in the symmetric group on $p+q$ letters and they can be assigned a sign $\sgn(\mu,\sigma)=\pm 1$ according to the sign of the associated permutation. 
		\end{remark}
		
		\begin{defn} \label{simple shuffles}
			We call a $(q,p)$ shuffle $(\mu,\sigma)$ \emph{simple} if there are  $a,b,c\in \underline{p+q}$ so that:
			\begin{enumerate}
				\item $\sigma(a)=0$.
				\item $\sigma \circ [a, a+1,\dots,b]$ is non-degenerate.
				\item $\sigma(i)=\sigma(b)$ for all $b\leq i\leq c$.
				\item $\sigma\circ [c,c+1,\dots,p+q]$ is non-degenerate.
			\end{enumerate} 
		\end{defn}
		This definition is not standard.  The simple $(q,p)$ shuffles are essentially those for which $\sigma$ can be decomposed into four parts: two ``degenerate'' parts and two ``non-degenerate'' parts.  The shuffle of display \eqref{shuffle 1} is not simple.  An example of a simple $(4,5)$ shuffle is 
		\[
			(\mu,\sigma)=([0,1,1,1,1,2,3,4,4,4],[0,0,1,2,3,3,3,3,4,5])
		\] 
		where $a=1$, $b=4$ and $c=7$. 
		\par
		Our interest in $(q,p)$ shuffles stems from their appearance in the Eilenberg--Zilber map.  We will have need of an explicit formula for this map later, so we fix notation here.  Let $X$ and $Y$ be topological spaces, the Eilenberg--Zilber map 
		\[
			\textbf{EZ}\colon C_{*}(X)\otimes C_{*}(Y)\to C_{*}(X\times Y)
		\]
		is given on a generator $\alpha\otimes\beta\in C_{q}(X)\otimes C_{p}(Y)$ by
		\[
			\textbf{EZ}(\alpha\otimes \beta) = \sum\limits_{(\mu,\sigma)}\sgn(\mu,\sigma)(\alpha\circ \mu)\times (\beta\circ \sigma)
		\]
		where the sum is over all $(q,p)$ shuffles.  
		\begin{remark}
			Here we are abusing notation. The $(q,p)$ shuffles are morphisms in the simplex category but we are identifying them with their geometric realizations, $|\mu|\colon\Delta^{p+q}\to \Delta^q$ and $|\sigma|\colon\Delta^{p+q}\to \Delta^p$.  We will continue to do this for the rest of the paper.
		\end{remark}

	\section{The Homology Coalgebra}\label{section: homology coalgebra}
		Let $X$ be a topological space. It is well known that, with coefficients in a field $k$, the singular homology of $X$ admits a comultiplication.  There are several ways of seeing this fact, but the most direct is to define it by the composition:
		\begin{equation} \label{usual comult}
			H_{*}(X)\to H_{*}(X\times X)\cong H_{*}(X)\otimes H_{*}(X)
		\end{equation}
		where the first map is induced by the diagonal of $X$ and the second map is the Kunneth isomorphism. For our purposes however, this will not be enough.
		\par 
		An important restriction of the Kunneth isomorphism is that it only applies to homology with coefficients in a field.  There are more general statements of the Kunneth theorem, but all require at least that the coefficients have a ring structure.  We will have need later of homology with coefficients in groups that are not naturally rings.  Specifically, we will be interested in the $E^2$ page of the Serre spectral sequence of a fibration; the coefficient groups will be the homology groups of the fiber.  To get a comultiplication in this context, we will make use of the fact that our coefficient groups are themselves $k$-coalgebras.
		\par 
		We begin by constructing the usual comultiplication in homology using the Alexander--Whitney map.  After giving the details of the usual construction, we modify it slightly to accommodate coefficients in coalgebras over $k$. 
		\par 
		A generator of $C_n(X\times X)$ is given by the product of two generators $\alpha$ and $\beta$ of $C_n(X)$.  The Alexander--Whitney map $\mathbf{AW}\colon C(X\times X)\to C(X)\otimes C(X)$ is given on generators by 
		\[
			\mathbf{AW}(\alpha\times \beta) = \bigoplus\limits_{i=0}^n (\alpha\circ [0,1,\dots,j])\otimes (\beta\circ [j,\dots,n])
		\]
		where $[0,\dots,j]\colon|\Delta^j|\to |\Delta^n|$ is the geometric realization of the map in $\mathbf{\Delta}$.
		\par
		The coalgebra structure $\triangle\colon C(X)\to C(X)\otimes C(X)$ is the composition
		\[
			\begin{tikzcd}
			C(X)	\arrow[r] & C(X\times X)\arrow[r,"\mathbf{AW}"] & C(X)\otimes C(X),
			\end{tikzcd}
		\]
		where the first map is induced by the diagonal.  
		\begin{lem} \label{chain comult}
			The map $\triangle\colon C(X)\to C(X)\otimes C(X)$ defined above is coassociative. Furthermore, if $X$ is connected then  the augmentation map $\epsilon\colon C(X)\to k$ makes $(C,\triangle,\epsilon)$ into a coassociative, counital coalgebra.
		\end{lem}
		\begin{proof}
			We check coassociativity, the rest is straightforward.  To make the notation easier, let $d_i,\delta_i\colon|\Delta^{n-i}|\to |\Delta^{n}|$ be the maps:
			\begin{align*}
				\delta_i & = [0,1,\dots,i] \\
				d_i & = [i,i+1,\dots,n]
			\end{align*}
			
			Let $\alpha\in C_n(X)$ be a generator.  Then $\triangle(\alpha)$ is given by 
			\begin{equation}\label{temp 1}
				\bigoplus\limits_{i=0}^n \delta_{i}^*(\alpha)\otimes d_{n-i}^*(\alpha)
			\end{equation}
			Applying $\triangle\otimes 1$, we get
			\[
				\bigoplus\limits_{i=0}^n\bigoplus\limits_{j=0}^{n-i} \delta_{i+j}^*(\alpha)\otimes (d_{n-(i+j)}^*\circ \delta_i^*)(\alpha)\otimes d_{n-i}^*(\alpha)
			\]
			Substituting $r=i+j$, we get:
			\[
				\bigoplus\limits_{i=0}^n\bigoplus\limits_{r=i}^{n} \delta_{r}^*(\alpha)\otimes (d_{n-r}^*\circ \delta_i^*)(\alpha)\otimes d_{n-i}^*(\alpha)
			\]
			which is the same as
			\[
				\bigoplus\limits_{r=0}^n\bigoplus\limits_{i=0}^{r} \delta_{r}^*(\alpha)\otimes (d_{n-r}^*\circ \delta_i^*)(\alpha)\otimes d_{n-i}^*(\alpha)
			\]
			This is (up to index variable names) what we get when we apply $1\otimes\triangle$ to display \eqref{temp 1}.
		\end{proof}
		
		It is not hard to check that $\triangle$ commutes with the relevant differentials, so it descends to a map on homology and recovers homology comultiplication of composition \eqref{usual comult}:
		\[
				H(C(X))\xrightarrow{\triangle_*} H(C(X)\otimes C(X)) \xrightarrow{\cong} H_*(X)\otimes H_*(X)
		\] 
		where the isomorphism is the Kunneth isomorphism.  Taking coefficients in a coalgebra $\Gamma$, the construction carries through in essentially the same way.
		\begin{proposition}\label{new comult}
			Let $X$ be a space, and let $\Gamma$ be a $k$-coalgebra with comultiplication $\triangle_{\Gamma}:\Gamma\to \Gamma\otimes \Gamma$.  Then $H_*(X;\Gamma)$ admits the structure of a $k$-coalgebra.
		\end{proposition}
		\begin{proof}
			By definition, $H_*(X;\Gamma)$ is the homology of the chain complex $C_*(X)\otimes \Gamma$ with boundary map $\partial = \partial_{C_*(X)}\otimes 1$.  This admits a comultiplication:
			\[
					C(X)\otimes \Gamma \xrightarrow{\triangle_X\otimes \triangle_{\Gamma}}
					C(X)\otimes C(X)\otimes \Gamma\otimes \Gamma 
					\xrightarrow{\widetilde{\otimes}}  
					(C(X)\otimes \Gamma) \otimes (C(X)\otimes \Gamma)
			\]
			where $\triangle_X$ is the map of Lemma \ref{chain comult} and $\widetilde{\otimes}$ is given by
			\begin{equation}
				\label{twisty_tensor}
				(a\otimes b)\widetilde{\otimes}(c\otimes d) = (-1)^{|b|\cdot|c|}(a\otimes c)\otimes(b\otimes d) 
			\end{equation}
			
			Because the Alexander--Whitney map commutes with $\partial_{C_*(X)}$, it is immediate that this commutes with $\partial$ and so descends to a comultiplication on homology.  This comultiplication can be identified using the universal coefficient theorem as
			\[
				H_*(X;k)\otimes \Gamma \xrightarrow{\triangle_X\otimes \triangle_{\Gamma}}
				H_*(X;k)\otimes H_*(X;k)\otimes \Gamma\otimes \Gamma 
				\xrightarrow{\widetilde{\otimes}}  
				H_*(X;k)\otimes \Gamma \otimes H_*(X;k)\otimes \Gamma
			\]
		\end{proof}
		Going forward we will abbreviate this comultiplication by $\triangle_X\twisty\triangle_{\Gamma}$.
		
		\begin{notn}\label{tensor notations}
			It becomes clear in the proof of the above proposition that we will need more notation for tensor products.  In particular, if we have two coalgebras $H$ and $\Gamma$ with comultiplications $\triangle_{H}$ and $\triangle_{\Gamma}$ there are really three different kinds of tensor products at play. 
			\begin{enumerate}
				\item The tensor product of elements in $H$ with elements in $\Gamma$ is denoted by $\square$.  That is, if $x\in H$ and $y\in \Gamma$, the associated simple tensor in $H\otimes \Gamma$ is denoted $x\squarebin y$.
				\item The tensor product of elements in $H$ with elements in $H$ is denoted by $\wedge$.  That is, if $x,y\in H$, the associated simple tensor in $H\otimes H$ is denoted $x\wedge y$.  The $\wedge$ is also used for elements of $\Gamma\otimes \Gamma$.
				\item The tensor product of elements in $H\otimes \Gamma$ with elements in $H\otimes \Gamma$ is denoted by $\otimes$.  That is, if $x\squarebin y,w\squarebin z\in H\otimes \Gamma$, the associated simple tensor in $(H\otimes \Gamma)\otimes (H\otimes \Gamma)$ is denoted $(x\squarebin y)\otimes(w\squarebin z)$.
			\end{enumerate}
		\end{notn}
		For example, with these notations we rewrite the formula \eqref{twisty_tensor} for $\widetilde{\otimes}$ as:
		\begin{equation}\label{twisty_2}
			(a\wedge b)\twisty(c\wedge d) = (-1)^{|b|\cdot|c|}(a\squarebin c)\otimes (b\squarebin d)
		\end{equation}

	\section{The Serre Spectral Sequence}\label{section: SSS}
		Recall that a map of spaces $E\to B$ is called a Serre fibration if for all CW complexes $X$ and homotopies $h\colon X\times I\to B$ such that $h_0\colon X\times \{0\}\to B$ factors through $E$, there is a lift $\overline{h}\colon X\times I\to E$.  Given a fibration $f\colon E\to B$, and a point $b\in B$, the fiber over the point $b$ is the space $F_b=f^{-1}(b)$.  When the space $B$ is path connected, all the fibers over the various points are homotopy equivalent and we will simply refer to the fiber $F$.  In his thesis \cite{Serre3}, Serre proved the following:
		\begin{thm} [Serre]\label{SSS}
			Let $G$ be an abelian group.  Given a fibration $F\to E\to B$, with $B$ a simply connected CW complex, there is a spectral sequence $\{E^r,d^r\}$ converging to $H_*(E;G)$, with the second page given by $E^2_{p,q}\cong H_p(B;H_q(F;G)).$
		\end{thm}
		The goal of this section is to show that when $G=k$ is a field the homological Serre spectral sequence admits a comultiplication.  In fact, we will do a bit more and show that this comultiplication can be identified on the $E^2$ page of the spectral sequence with the comultiplication of Proposition \ref{new comult}.  
		\par
		Toward this, let $\pi\colon E\to B$ be a fibration with $B$ a simply connected CW complex and fiber $F$. Let $k$ be a field that we fix for the remainder of the section. We consider the graded vector space $C=C_{*}(E)$ of singular $k$-chains in $E$ with boundary map $\partial$.  Following Serre \cite[Section II $n^{\circ}4$]{Serre3}, we define a filtration on $C$.  We write $\Delta^n$ for the geometric $n$ simplex.
		\begin{defn} \label{i-degenerate}
			Let $\alpha \colon \Delta^n\to E$.  We say that $\alpha$ is \emph{$i$-degenerate} with respect to $\pi$ if $\pi\circ \alpha = \alpha'\circ s$ for some codegeneracy map $s\colon \underline{n}\to \underline{n-i}$ and some $\alpha'$.  That is, $\alpha$ is $i$-degenerate if the following diagram commutes.
			\[
			\begin{tikzcd}
				\Delta^n \arrow[r,"\alpha"] \arrow[d,"s"']	&	E \arrow[d,"\pi"]	\\
				\Delta^{n-i}	\arrow[r,"\alpha'"]	&	B
			\end{tikzcd}
			\]
		\end{defn}
		For $0\leq q\leq n$, let $T^{p,q}\subset C_{p+q}$ be the subspace generated by the simplices that are $q$-degenerate with respect to $\pi$.  For $q\leq 0$ define $T^{p,q}=T^{p,0}$ and for $g> n$ let $T^{p,q}=0$. The Serre filtration of $C$ is given by
		\[
			F^p = \bigoplus\limits_{q}T^{p,q}
		\]
		With this indexing we have:
		\[
			T^{p,q}=F^p\cap C_{p+q} = F^{p}_{p+q}
		\]
		From the definition we have $T^{p,q}\subset T^{p+1,q-1}$, so we get for free that $F^p\subset F^{p+1}$.  Furthermore, one can check this is a first quadrant filtration in the sense of Definition \ref{first quadrant filtration}.  One confirms directly from the definitions that 
		\[
			\partial(T^{p,q})\subset T^{p-1,q}+T^{p,q-1}\subset T^{p,q-1}
		\]
		so we have $\partial(F^p)\subset F^p$ and so $(C,\partial,\{F^p\})$ is an object in $\FDGM_k$.  The associated spectral sequence $\textbf{E}(C)$ is the homological Serre spectral sequence.
		 
		\begin{lem} \label{coproduct_respects_filtration}
			The coproduct $\triangle\colon C\to C\otimes C$ defined in Lemma \ref{chain comult} is a map of filtered complexes where the filtration on $C\otimes C$ is induced by the Serre filtration. 
		\end{lem}
		\begin{proof}
			Let $\alpha\in T^{p,q}\setminus T^{p-1,q+1}$ be a generator of $F^p$. In the notation introduced in Section \ref{section: fdgcs}, we show that $\triangle(\alpha)\in G^p$.  By definition, we have
			\[
				\triangle(\alpha) = \bigoplus\limits_{i=0}^{p+q} (\alpha\circ [0,\dots,i])\otimes (\alpha\circ [i,\dots,p+q])
			\]
			so we just need to check the degeneracy of $\pi \circ \alpha\circ [0,\dots,i]$ and $\pi \circ \alpha\circ [i,\dots,p+q]$.
			
			Because $\alpha\in T^{p,q}$ there is a codegeneracy $s\colon \underline{p+q}\to \underline{p}$ so that $\pi\circ \alpha = \alpha'\circ s$. Since we assumed $\alpha\notin T^{p-1,q+1}$ we have $\alpha'$ is not degenerate, so we just need to study the degeneracy of $s\circ [0,\dots,i]$ and $s\circ [i,\dots,p+q]$.  This can be done in the simplex category.
			 
			Because $s$ is a codegeneracy it is surjective so $\#\Ima(s\circ [0,\dots,i]) = s(i)+1$.  Similarly, $\Ima(s\circ [i,\dots,p+q]) = s(p+q)-s(i)+1 = p-s(i)+1$.  We can factor the two maps of interest as $d^i_1\circ s^i_1$ and $d^i_2\circ s^i_2$ where $s^i_1\colon \underline{i}\to \underline{s(i)}$ and $s^i_2\colon \underline{p+q-i}\to \underline{p-s(i)}$ are codegeneracies.  Thus we can rewrite 
			\begin{align*}
				\pi\circ \alpha\circ [0,\dots,i] & = \alpha'\circ d_1^i\circ s_1^i \\
				\pi\circ \alpha\circ [i,\dots,p+q] & = \alpha'\circ d_2^i \circ s_2^i,
			\end{align*}
			so $\alpha\circ [0,\dots,i]\in T^{s(i),i-s(i)}$ and $\alpha\circ [i,\dots,p+q] \in T^{p-s(i),q-i+s(i)}$ and $\triangle(\alpha)\in G^p$. 
		\end{proof}
		Applying Theorem \ref{fdgcs gives spectral sequence of coalgebras} now gives
		\begin{cor} \label{has-spec}
			The homological Serre spectral sequence associated to a fibration $F\to E\to B$ is a spectral sequence of coalgebras that converges to $H_*(E)$ as a coalgebra.
		\end{cor}
		We now describe how to identify the $E^2$ page of the spectral sequence, deferring some of the details to McCleary's book \cite{Mc}.
		 
		The construction begins with a slightly modified version of the singular chain complex with coefficients.  Let $*\in \Delta^n$ be the basepoint we get from the coface map $[0]\colon \Delta^0\to \Delta^n$.  Define $C_p(B;C_q(F))$ to be the subspace of $C_q(E)\otimes C_p(B)$  generated by simple tensors $U\otimes V$ with $V\colon \Delta^p\to B$ and $U\colon \Delta^q\to E$ such that the image of $U$ is contained in $F_{V(*)}$.  Here we take $C_p(B)$ and $C_q(E)$ to be the normalized chain groups, obtained from the usual singular chain groups by quotienting out the degenerate simplices.  We make $C_p(B;C_q(F))$ a chain complex with the boundary map $\partial = \partial_E\otimes 1$ where $\partial_E$ is the boundary map of $C_{*}(E)$.
		
		The following lemma follows from Propositions 5.20 and 5.22 of \cite{Mc}. 
		\begin{lem} \label{local system}
			The chain group $C_*(B;C_*(F))$ is isomorphic to $C_*(B;H_*(F))$, the usual singular complex of $B$ with coefficients in the coalgebra $H_*(F)$.  Here $H_*(F)$ is the homology of a single fixed fiber $F$ of $\pi\colon E\to B$.
		\end{lem}

		We define a map $\phi^{p,q}\colon T^{p,q}\to C_p(B;C_q(F))$ on generators $\alpha$ by
		\[
			\phi^{p,q}(\alpha) = (\alpha\circ [0,\dots,q])\otimes (\pi\circ \alpha\circ [q,\dots,p+q])
		\]
		\begin{proposition}
			The map $\phi^{p,q}\colon T^{p,q}\to C_p(B;C_q(F))$ is well defined.  Furthermore, $\phi^{p,q}(T^{p-1,q+1})=0$, so this map descends to a map $\phi^{p,q}:E^0_{p,q}\to C_p(B;C_q(F))$
		\end{proposition}
		\begin{proof}
		Let $\alpha \in T^{p,q}$.  By definition there is some $\alpha'\colon \Delta^{p}\to B$ so that $\pi\circ \alpha = \alpha'\circ s$ for some codegeneracy map $s\colon \Delta^{p+q}\to \Delta^p$.  Since $s$ is the realization of codegeneracy map, we can write $s=[s_0,s_1,\dots,s_{p+q}]$ for $0\leq s_i\leq p$.  
		
	 	We prove the second claim first.  Note that $\phi^{p,q}(\alpha)=0$ if 
	 	\[
	 		\pi\circ \alpha\circ [q,\dots,p+q] = \alpha'\circ s\circ [q,\dots,p+q] = \alpha' \circ [s_q,\dots,s_{p+q}]
	 	\]
	 	 is degenerate.   We see that $[s_q,\dots,s_{q+p}]$ is degenerate unless $s=[0,0,\dots,0,1,2,\dots,p]$.  When $s=[0,0,\dots,0,1,2,\dots,p]$, the map $\alpha' \circ [s_q,\dots,s_{p+q}]$ is degenerate if and only if $\alpha'$ is degenerate.  Since $\alpha'$ being degenerate implies $\alpha\in T^{p-1,q+1}$, we have $\phi^{p,q}(\alpha)=0$ unless $s=[0,\dots,0,1,\dots, p]$ and $\alpha\notin T^{p-1,q+1}$.  This proves the second claim.
	 	 
	 	 To prove the map is well defined we need to see that $\phi^{p,q}(\alpha)$ is actually an element of $C_p(B;C_q(F))$.  That is, we need to check that  either $\phi^{p,q}(\alpha)=0$ or the image of $\pi\circ \alpha\circ [0,\dots, q]$ is equal to
	 	 \[
	 	 	\pi\circ \alpha\circ[q,\dots,p+q]\circ[0] = \alpha'(s([q]))
	 	 \]
	 	When $\phi^{p,q}(\alpha)\neq 0$, we have shown $s=[0,\dots,0,1,\dots, p]$ so we have 
	 	\[
	 		\alpha'(s([q])) = \alpha'(*)
	 	\]  
	 	On the other hand,
		\[
			\pi\circ \alpha\circ [0,\dots,q] = \alpha' \circ s\circ [0,\dots,q] = \alpha'\circ [0,\dots,0],
		\]
		and the image of this simplex is obviously $\alpha'(*)$.
		\end{proof}
	
		We now describe a map going in the opposite direction that will be a chain inverse to $\phi$.  Let $U\otimes V$ be a generator of $C_p(B;C_q(F))$.  We first build a map $G=G_{U,V}:\Delta^q\times\Delta^p\to E$ by the following construction.  Given $y\in \Delta^p$, fix a line segment $L_y\colon I\to \Delta^p$ such that $L_y(0)=*$ and $L_y(1)=y$.  In the case $y=*$ we always pick the constant map.  For any $x\in \Delta^q$ we know that $\pi(U(x))=V(*)$.  Since $\pi$ is a fibration we can lift $V\circ L_y$ to a path $L_y^x:I\to E$ so that $L_y^x(0)=U(x)$ and $\pi(L_y^x(1))=V(y)$.  Again, when $y=*$ we let $L_y^x$ be the constant map at $U(x)$.  We define $G(x,y)=L_y^x(1)$.  
		\begin{remark}
			If one is more careful with the lifting property for the fibration $\pi\colon E\to B$ it is possible to construct the $L_y^x$ in such a way that the map $(x,y)\mapsto L_y^x$ is a continuous map from $\Delta^q\times \Delta^p$ to the free path space $E^I$.  Using this, it is immediate that the map $G$ constructed is in fact continuous and we will use this presently.  Since we will not have any need of the finer details of the construction just described we refer the reader to \cite[Proposition 4.26]{Mc}.
		\end{remark}
		
		\begin{lem} \label{computing G}
			For all $x\in \Delta^p$ and all $U\otimes V\in C_p(B;C_q(F))$ we have $G_{U,V}(x,*)=U(x)$.
		\end{lem}
		\begin{proof}
			This is the special case where $y=*$ is the basepoint.  Since we picked $L_*^x\colon I\to E$ to be the constant map at $U(x)$ we have $G_{U,V}(x,*) = L_*^x(1) = U(x)$.
		\end{proof}
		
		\par
		
		Now, for each $n$ let $1_n\in C_n(\Delta^n)$ be the identity map on $\Delta^n$.  We define $\psi\colon C_p(B;C_q(F))\to E^0_{p,q}=T^{p,q}/T^{p-1,q+1}$ on generators $U\otimes V$ to be the class represented by applying the composite 
		\[
			C_q(\Delta^q)\otimes C_p(\Delta^p)\xrightarrow{\textbf{EZ}}C_{p+q}(\Delta^q\times\Delta^p)\xrightarrow{G_*} C_{p+q}(E)
		\]
		to the element $1_q \otimes 1_p$, where $G_*$ is induced by $G_{U,V}$ and $\mathbf{EZ}$ is the Eilenberg--Zilber map. We will abuse notation and just write $\psi(U\otimes V) = G_*(\textbf{EZ}(1_q\times 1_p))$ letting context dictate whether we are working with an element in $T^{p,q}$ or a class in $E^0_{p,q}$.
		
		To see that this map is well defined we must check that the representing element is actually in $T^{p,q}$.  Note that $\textbf{EZ}(1_q\otimes 1_p)$ is just the sum over all $(q,p)$ shuffles.  This allows us to rewrite $\psi(U\otimes V)$ as
		\[
			\psi(U\otimes V) = \sum\limits_{(\mu,\sigma)}\sgn(\mu,\sigma) G_*(\mu,\sigma)
		\]  
		where the sum runs over all $(q,p)$ shuffles.  That $\psi$ is well defined now follows immediately from the following lemma.
		\begin{lem} \label{diagram_lemma}
			Let $U\otimes V\in C_p(B;C_q(F))$, $G=G_{U,V}$, and let $(\mu,\sigma)$ be any $(q,p)$ shuffle.  Then $\pi\circ G\circ (\mu,\sigma)=V\circ \sigma:\Delta^{p+q}\to B$.
		\end{lem}
		\begin{proof}
			The lemma follows from the following commutative diagram:
			\[
			\begin{tikzcd}[column sep = large, row sep = large]
				\Delta^{q+p}\arrow[r,"\mu\times \sigma"] \arrow[d,"\sigma"] & \Delta^q\times \Delta^p \arrow[r,"G"] \arrow[d] & E \arrow[d,"\pi"] \\
				\Delta^p \arrow[r,"="]	&	\Delta^p \arrow[r,"V"] 	& B
			\end{tikzcd}
			\]
			where the middle arrow is the projection.  The left square obviously commutes while the right square commutes because for any $(x,y)\in \Delta^q\times \Delta^p$ we have: $\pi(G(x,y)) = \pi(L_y^x(1))=V(y)$.
		\end{proof}
		McCleary \cite[Lemmas 5.24--5.28]{Mc} shows that the maps $\phi$ and $\psi$ induce the usual identification of the $E^2$ page of the Serre spectral sequence.
		\begin{proposition} \label{mccleary prop}
			The maps $\phi^{p,q}\colon E^0_{p,q}\to C_p(B;C_q(F))$ fit into a map $\phi\colon E^0_{*,*}\to C_*(B;C_*(F))$ of bigraded complexes with chain inverse $\psi$.  Furthermore, the induced isomorphism $H(\phi)\colon E^1_{*,*}\to C_*(B;H_*(F))$ is also a map of chain complexes and therefore descends to an isomorphism of bigraded complexes $E^2_{p,q}\to H_p(B;H_q(F))$.
		\end{proposition}
		 
		 Our goal is to use the maps $\phi$ and $\psi$ to put a comultiplication on $C_*(B;C_*(F))$ that descends to all pages of the spectral sequence after identifying the $E^2$ page with $H_*(B;H_q(F))$.  This comultiplication will have the advantage on the $E^2$ page it is exactly the comultiplication of Proposition \ref{new comult}.
		 \par
		 As a first step, define $\nabla^0$ in the below diagram by composition of the other maps.	 
		 \begin{equation}\label{big diagram}
		 	\begin{tikzcd}[column sep = huge]
			 	E^0_{p,q} \arrow[d,"\triangle^0"']
			 	 & C_p(B;C_q(F)) \arrow[l,"\psi"'] \arrow[d,"\nabla^0"] \\
			 	\bigoplus\limits_{\substack{a+b=p\\c+d=q}}E^0_{a,c}\otimes E^0_{b,d} \arrow[r,"\phi\otimes \phi"] & \bigoplus\limits_{\substack{a+b=p\\c+d=q}}C_a(B;C_c(F))\otimes C_b(B;C_d(F))
		 	\end{tikzcd}
		 \end{equation}
		 
		Here $\triangle^0$ is the comultiplication of Corollary \ref{has-spec}. In the remainder of this section we compute $\nabla^0(U\otimes V)$ for a generator $\mathit{U\otimes V}\in C_p(B;C_q(F))$ and use this to determine the comultiplication $\nabla^2$ on the $E^2$ page.  We assume neither $U$ nor $V$ is degenerate.
		\par
		The argument is fairly technical, so before getting into it we give a brief outline.  Since $\psi$ is constructed using the Eilenberg--Zilber map the left-down composite of diagram \ref{big diagram} will be a rather large sum of elements of the form $\triangle^0 G_*(\mu,\sigma)$ where $(\mu,\sigma)$ is a $(q,p)$ shuffle.  The key step in the proof is showing that $\phi\otimes \phi$ applied to such terms will be zero unless $(\mu,\sigma)$ is a simple shuffle in the sense of Definition \ref{simple shuffles}.  There are significantly fewer simple shuffles than there are total shuffles and so this reduction gives us enough control on the sum to compute it directly.
		\par
		To simplify the notation we assume that the field $k$ is characteristic 2.  This allows us to ignore minus signs when dealing with both the $\widetilde{\otimes}$ operator  of equation \eqref{twisty_2} and the Eilenberg--Zilber map.  One can check without too much trouble that the signs we are omitting do indeed work out when the characteristic is not 2.  Indeed, the two kinds of signs we are omitting exactly cancel with one another.
		 
		Computing directly from the definitions we have
		\[
			(\triangle \circ \psi)(U\otimes V) = \sum\limits_{(\mu,\sigma)}\sum\limits_{j=0}^{p+q}G_*(\mu\circ [0,\dots,j],\sigma\circ [0,\dots,j])\otimes G_*(\mu \circ [j,..,p+q],\sigma\circ [j,\dots,p+q])
		\]
		
		Writing $X^{\mu,\sigma}_j=G_*(\mu\circ [0,\dots,j],\sigma\circ [0,\dots,j])$ and $Y^{\mu,\sigma}_j= G_*(\mu \circ [j,..,p+q],\sigma\circ [j,\dots,p+q])$ we can obtain the more manageable form:
		\begin{equation}\label{the_sum}
			(\triangle \circ \psi)(U\otimes V) = \sum\limits_{(\mu,\sigma)}\sum\limits_{j=0}^{p+q}X_j^{\mu,\sigma}\otimes Y_j^{\mu,\sigma}
		\end{equation}
		 
		To apply the bottom map in diagram \eqref{big diagram} we must first identify which direct summands $E^0_{a,c}\otimes E^0_{b,d}$ of the lower left vector space  the elements $X_j^{\mu,\sigma}\otimes Y_j^{\mu,\sigma}$ live in.  
		\begin{lem}
			For any $(q,p)$ shuffle $(\mu,\sigma)$, and any $j$, we have $X_j^{\mu,\sigma}\in T^{\sigma(j),j-\sigma(j)}$ and $Y_j^{\mu,\sigma}\in T^{p-\sigma(j),q-j+\sigma(j)}$.
		\end{lem}  
		\begin{proof}
			Using Lemma \ref{diagram_lemma}, we know that 
			\[
				\pi_*(X_j^{\mu,\sigma}) = V\circ \sigma\circ[0,\dots,j]
			\]
			Since $V$ is not degenerate, it suffices to compute the degeneracy of $\sigma\circ [0,\dots,j]$ as a map in the simplex category.  This is very similar to the counting in the proof of Lemma \ref{coproduct_respects_filtration}.
			 
			Because $\sigma:\underline{p+q}\to \underline{p}$ is a surjective map, we must have $\sigma\circ[0,\dots,j]$ is a $j$ simplex that takes exactly $\sigma(j)+1\leq j+1$ values.  It follows that $V\circ \sigma\circ[0\dots j]$  is $(j-\sigma(j))$-degenerate and therefore $G_*(\mu\circ [0,\dots,j],\sigma\circ [0,\dots,j])\in T^{\sigma(j),j-\sigma(j)}$.  The computation for $Y_j^{\mu,\sigma}$ is similar.  
		\end{proof}
		
		With the degeneracies in hand we proceed to show the only terms $\phi\otimes \phi$ does not take to zero correspond to the simple $(q,p)$ shuffles of Definition \ref{simple shuffles}.  As a first step, we have
		\begin{proposition} \label{first_prop}
			For a $(q,p)$ shuffle $(\mu,\sigma)$, let $a$ be the largest number so that $\sigma(a)=0$ and let be the largest number so that $\sigma(b)=b-a$. Then $(\phi^{\sigma(j),j-\sigma(j)})(X_j^{\mu,\sigma})$ is zero for all $j>b$. 
		\end{proposition}
		\begin{remark}\label{temp remark 1}
			The $a$ and $b$ in this statement are maximal choices of $a$ and $b$ satisfying $(1)$ and $(2)$ of Definition \ref{simple shuffles}.
		\end{remark}
	
		\begin{proof}
			 By surjectivity of $\sigma$ we always have $\sigma(a+1)=1=a+1-a$ so we have $b>a$.  By definition we have 
			\[
				\phi^{\sigma(j),j-\sigma(j)}(X_j^{\mu,\sigma}) = (X_j^{\mu,\sigma}\circ [0,\dots,j-\sigma(j)])\otimes (\pi\circ X_j^{\mu,\sigma}\circ [j-\sigma(j),\dots,j])
			\]
			To show this is zero, it is enough to show that $\pi\circ X_j^{\mu,\sigma}\circ [j-\sigma(j),\dots,j]$ is a degenerate simplex in $B$.  By Lemma \ref{diagram_lemma} we know 
			\[
				\pi\circ X_j^{\mu,\sigma}\circ [j-\sigma(j),\dots,j] = V\circ \sigma \circ [0,\dots,j]\circ [j-\sigma(j),\dots,j] = V\circ \sigma\circ [j-\sigma(j),\dots,j], 
			\]
			so it suffices to show $\sigma\circ [j-\sigma(j),\dots,j]$ is a codegeneracy whenever $j>b$.  Equivalently, we will show that some value in $\{\sigma(j-\sigma(j)),\dots, \sigma(j) \}$ is repeated.
			
			We first prove the result for $j=b+1$ then proceed inductively.  Since $\sigma$ is surjective and monotone we must have $\sigma(b+1)=b-a$ or $\sigma(b+1)=b-a+1$.  As the later case is impossible by the maximality of $b$, we have $\sigma(b+1)=b-a$.  Thus $b+1-\sigma(b+1)=a+1$ and
			\[
				\sigma\circ [b+1-\sigma(b+1),\dots,b+1] = \sigma\circ [a+1,\dots,b,b+1]
			\]
			This is degenerate because $\sigma(b)=b-a=\sigma(b+1)$. We have used here that $a<b$ so that $a+1\leq b$.
			
			We now work inductively.  Suppose the result is true for $j=b+1, b+2,\dots,k$.  There are two possibilities: $\sigma(k+1)=\sigma(k)$ or $\sigma(k+1)=\sigma(k)+1$.  In the first case, we have 
			\[
				k+1-\sigma(k+1)=k-\sigma(k)+1< k+1,
			\]
			so $k+1-\sigma(k+1)\leq k$ and $\sigma\circ [k+1-\sigma(k+1),\dots,k,k+1]$ must be degenerate since $\sigma(k)=\sigma(k+1)$.
			
			In the case $\sigma(k+1)=\sigma(k)+1$, we have $k+1-\sigma(k+1)=k-\sigma(k)$ and 
			\[
				\sigma\circ [k+1-\sigma(k+1),\dots,k+1]=\sigma\circ [k-\sigma(k),\dots,k+1]
			\]
			is degenerate because  $\sigma\circ [k-\sigma(k),\dots,k]$ is degenerate by induction.
		\end{proof}
		Similarly, we have:
		\begin{proposition}\label{second_prop}
			For $(q,p)$ shuffle $(\mu,\sigma)$, let $c$ be the smallest number such that $\sigma\circ [c,\dots,p+q]$ is non-degenerate and let $b<c$ be the smallest number so that $\sigma(b)=\sigma(c)$. Then $(\phi^{p-\sigma(j),q-j+\sigma(j)})(Y_j^{\mu,\sigma})$ is zero for all $j<b$.
		\end{proposition}
		\begin{remark}\label{temp remark 2}
			The $b$ and $c$ in this statement are minimal choices of $b$ and $c$ satisfying the last two conditions of the definition of a simple $(q,p)$ shuffle.
		\end{remark}
		\begin{proof}
			This time it suffices to show that 
			\[
				\sigma\circ [j,\dots,p+q]\circ [q-j+\sigma(j),\dots,p+q-j] = \sigma\circ [q+\sigma(j),\dots,p+q]
			\]
			is degenerate for all $j<b$.  When $j=b-1$, we have $\sigma(b-1)=\sigma(b)-1$, so $q+\sigma(b-1) = q+\sigma(b)-1 = q+\sigma(c)-1$.  It is not too hard to check that the non-degeneracy condition on $c$ implies that $p+q-c=p-\sigma(c)$.  Thus $q+\sigma(c)=c$, so $q+\sigma(b-1)=c-1$ and the result follows because by the minimality of $c$ we must have $\sigma\circ [c-1,\dots,p+q]$ is degenerate. 
			\par 
			Now suppose the result is true for $b-1,b-2,\dots,k$.  There are two options, $\sigma(k-1)=\sigma(k)-1$ and $\sigma(k-1)=\sigma(k)$.  In either case the result follows because $q+\sigma(k+1)\leq q+\sigma(k)$ and we assumed the result for $k$.
		\end{proof}
		
		Putting the last two propositions together gives us what we are after.
		\begin{proposition} \label{not simple gives zero}
			If $(\mu,\sigma)$ is not a simple $(q,p)$ shuffle then $(\phi\otimes \phi)(X_j^{\mu,\sigma}\otimes Y_j^{\mu,\sigma})=0$
		\end{proposition}
		\begin{proof}
			Let $a$ and $b_1$ be as in Proposition \ref{first_prop} and let $b_2$ and $c$ be as in Proposition \ref{second_prop}.  Now suppose that $\sigma\neq [0,\dots,0,1,\dots,p]$.  Since this is a simple shuffle it will not affect our result but we need to exclude this shuffle because it is the only one for which $c=a$.  Assuming $c\neq a$ we claim that $b_1\leq b_2$.
			\par
			Suppose that $b_1>b_2$.  By the definition of $b_1$, $\sigma\circ [a,\dots,b_2,\dots,b_1]$ must be non-degenerate.  Then it must be that $b_2=c$ because otherwise $b_2+1\leq c,$ implying that $\sigma(b_2)=\sigma(b_2+1)=\sigma(c)$ making $\sigma\circ [a,\dots,b_2,..,b_1]$ degenerate.  On the other hand, if $b_2=c$ then because $\sigma\circ [c,\dots,p+q]$ is non-degenerate we would have $\sigma\circ [a,\dots,p+q]$ is non-degenerate.  But then the definition of $c$ implies that $c=a$, a contradiction of our assumption.
			\par
			It follows from Remarks \ref{temp remark 1} and \ref{temp remark 2} that a $(q,p)$ shuffle, except the exceptional one with $\sigma=[0,\dots,0,1,\dots,p]$, is simple if and only if $b_1=b_2$.  If $b_1<b_2$, then for any $j$ we must have $j>b_1$ or $j<b_2$ and the result follows from Propositions \ref{first_prop} and \ref{second_prop}.
		\end{proof}
		
		For the rest of this section we will reserve the letters $a$, $b$ and $c$ for their meaning in the context of simple $(q,p)$ shuffles. The key technical result is  the next corollary, but first we need some notation.  For any $y\in \Delta^p$ recall that we have fixed line segments $L_y\colon I\to \Delta^p$ so that $L_y(0)=*$ and $L_y(1)=y$.  These induce, in the standard way, homotopy equivalences of the fibers $F_{V(*)}\cong F_{V(y)}$.  In the case where $y$ is a vertex of $\Delta^p$, that is $y\in \{0,1,\dots,p\}$, we will denote the resulting isomorphisms of chain groups by $\alpha_y\colon C_{*}(F_{V(*)})\to C_{*}(F_{V(y)})$.  Because we have assumed that $B$ is simply connected, these isomorphisms result in canonical identifications in homology.
		
		We will consistently make use of Notation \ref{tensor notations} and the reader is encouraged to review this before continuing with the rest of the section.
		\begin{cor} \label{standard_case}
			Let $(\mu,\sigma)$ be a simple $(q,p)$ shuffle with $\sigma\neq [0,\dots,0,1,\dots, p]$.  Then
			$
				(\phi\otimes\phi)(X_j^{\mu,\sigma}\otimes Y_j^{\mu,\sigma})
			$ 
			is zero unless $j=b$.  When $j=b$, 
			\[
			 	(\phi\otimes\phi)(X_b^{\mu,\sigma}\otimes Y_b^{\mu,\sigma}) = (U\circ[0,\dots,a]\squarebin V\circ[0,\dots,b-a])\otimes (\alpha_{b-a}U\circ[a,\dots,q]\squarebin V\circ[b-a,\dots,p])
			\]
		\end{cor}
		\begin{proof}
			The first claim is immediate from the proof of the last proposition.  When $j=b$, we compute $\phi^{\sigma(b),a}(X_b^{\mu,\sigma})$ and $\phi^{p-\sigma(b),q-b+\sigma(b)}(Y_b^{\mu,\sigma})$.  For reference, we write $\sigma$ as:
			\[
				\sigma = [0,0,\dots,0,1,\dots,\sigma(b),\sigma(b),\dots,\sigma(b),\sigma(b)+1,\dots,p+q],
			\]
			where $c$ appears only implicitly as the last coordinate for which $\sigma(c)=\sigma(b)$.  Similarly, we write $\mu$ as:
			\[
				\mu = [0,1,\dots,a,a,\dots,a,a+1,\dots,q,q,\dots,q]
			\]
			Thus $\sigma\circ [0,\dots,b]=[0,\dots,0,1,\dots,\sigma(b)]$ and  $\sigma\circ [b,\dots,p+q]=[\sigma(b),\dots,\sigma(b),\sigma(b)+1,\dots,p+q]$. By definition:
			\[
				\phi^{\sigma(b),a}(X_b^{\mu,\sigma}) = (X_b^{\mu,\sigma}\circ[0,\dots,a])\otimes(\pi\circ X_b^{\mu,\sigma}\circ[a,\dots,b])
			\]
			where we have used the fact that $a+\sigma(b)=a+b-a=b$ to identify $[a,\dots,a+\sigma(b)]=[a,\dots,b]$. The right term in the tensor is, by Lemma \ref{diagram_lemma}, $V\circ \sigma\circ [a,\dots,b] = V\circ[0,\dots,b-a]$.  For the left term, we have:
			\begin{align*}
				X_b^{\mu,\sigma}\circ[0,\dots,a] & = G_{U,V}\circ(\mu\circ [0,\dots,b],\sigma\circ[0,\dots,b])\circ [0,\dots, a] \\
				& =  G_{U,V}\circ(\mu\circ [0,\dots,a],\sigma\circ[0,\dots,a]) \\
				& =	 G_{U,V}\circ([0,\dots,a],[0,\dots,0])
			\end{align*}
			Given a point $x\in \Delta^a$, we have 
			\[
			 	(X_b^{\mu,\sigma}\circ[0,\dots,a])(x) = (G_{U,V}\circ([0,\dots,a],[0,\dots,0]))(x) = G_{U,V}(x,*) = U(x)
			\]
			where the last equality is Lemma \ref{computing G}.  We conclude that $X_b^{\mu,\sigma}\circ[0,\dots,a]=U\circ [0,\dots,a]$, completing the description of $\phi^{\sigma(b),a}(X_b^{\mu,\sigma})$.
			\par
			The computation of $\phi^{p-\sigma(b),q-b+\sigma(b)}(Y_b^{\mu,\sigma})$ is completely analogous.  There is a minor difference in that the $\pi(U\circ[a,\dots,q])=V([0])$ which is a different basepoint than $V[b-a]$.  This is accounted for by the isomorphism $\alpha_{b-a}$.
		\end{proof}
	
		In the case $\sigma =[0,\dots,0,1,\dots,p]$ we will get a non-zero term for all $0\leq j\leq p+q$.  Other than this the computations are the same as the last proposition.
		\begin{cor}\label{exceptional_case}
			If $\sigma= [0,\dots,0,1,\dots, p]$ then for $j< q$ we have
			\[
				(\phi\otimes\phi)(X_j^{\mu,\sigma}\otimes Y_j^{\mu,\sigma}) = (U\circ[0,\dots,j]\squarebin V\circ[0])\otimes (U\circ[j,\dots,q]\squarebin V)
			\]
			For  $j\geq q$ we have
			\[
				(\phi\otimes\phi)(X_j^{\mu,\sigma}\otimes Y_j^{\mu,\sigma}) = (U\squarebin V\circ[0,\dots,j-q])\otimes (\alpha_{j-q}U\circ[j]\squarebin V[j-q,\dots,p])
			\]
		\end{cor}
	
		Finally we prove the main result.  
		\begin{thm}\label{the_hard_one}
		 	On the $E^2$ page of the Serre spectral sequence, the comultiplication $\nabla^2$ induced by the composition $\nabla^0$ agrees with the comultiplication of Lemma \ref{new comult} on $E^2_{*,*}\cong H_*(B;H_*(F))$.  
		\end{thm}
	 	\begin{proof}
	 		Fix a fiber $F$ of $\pi\colon E\to B$ and consider the function:
	 		\[
	 			C_p(B;C_q(F))\to C_q(F)\otimes C_p(B)
	 		\]
	 		given by $U\otimes V\mapsto \alpha_U(U)\otimes V$ where $\alpha_U$ is one of the isomorphisms from the discussion preceding Corollary \ref{standard_case}.  Lemma \ref{local system} implies that this map is the identity on homology.  It follows from Proposition \ref{mccleary prop} that when computing on the $E^2$ page of the spectral sequence we can ignore the $\alpha_x$'s from the previous two corollaries.  Instead, we identify all the chains with their images in the single chain group $C_{*}(B)\otimes C_{*}(F)$.  We use $\triangle_B$ and $\triangle_F$ to denote the comultiplications in $C_{*}(B)$ and $C_{*}(F)$ from Lemma \ref{usual comult}.
	 		\par
	 		We need notation for summing over all simple $(q,p)$ shuffles.  First, we note that simple $(q,p)$ shuffles are in some sense overdetermined.  Upon fixing values $a$ and $b$ one can determine that $c=b-a+q$.  It follows that to sum over all simple $(q,p)$ shuffles we need only sum over all possible values of $a$ and $b$.  As noted previously, we must have $a\leq q$.  For $b$, we have $c=b-a+q$, so we also have $b=c+a-q$.  Since $c\leq p+q$, we have $b\leq p+a$.  When $a\neq q$ one can check that that all values $a+1\leq b\leq a+p$ do indeed determine a simple $(q,p)$ shuffle.  The exceptional case $a=q$ corresponds to the exceptional shuffle $\sigma=[0,\dots,0,1,\dots,p]$; it is clear this is the only option when $a=q$. Since $a$ and $b$ determine $\mu$ and $\sigma$ we will write $X_{j}^{a,b}=X_j^{\mu,\sigma}$ and similarly for the $Y's$.
	 		\par
	 		  Now fix homology classes $[U]\otimes [V]\in H_p(B;H_q(F))$.  Using  Proposition \ref{not simple gives zero} and Corollary \ref{standard_case}, we have $\nabla^2([U]\otimes [V])$ is represented by
	 		\[
	 			R+\sum\limits_{a=0}^{q-1}\sum\limits_{b=a+1}^{a+p}((U\circ[0,\dots,a])\squarebin (V\circ[0,\dots,b-a]))\otimes ((U\circ[a,\dots,q])\squarebin (V\circ[b-a,\dots,p]))
	 		\]
	 		where $R$ is the contribution from the case $a=q$.  For fixed $a$ we have
	 		\begin{gather*}
	 			\sum\limits_{b=a+1}^{a+p}((U\circ[0,\dots,a])\squarebin (V\circ[0,\dots,b-a]))\otimes ((U\circ[a,\dots,q])\squarebin (V\circ[b-a,\dots,p])) \\
	 			 = [((U\circ[0,\dots,a])\wedge (U\circ[a,\dots,q])) \twisty \triangle_B(V)]-R'(a)
	 		\end{gather*}
	 		where 
	 		\begin{align*}
	 			R'(a) & = ((U\circ[0,\dots,a])\wedge (U\circ[a,\dots,q])) \twisty ((V\circ[0])\wedge V)\\
	 			& = (U\circ[0,\dots,a]\squarebin V\circ[0])\otimes (U\circ[a,\dots,q]\squarebin V)
	 		\end{align*}
	 		The $R'(a)$ terms are exactly the terms in $R$ that result from the cases $j< q$ in Corollary \ref{exceptional_case}.  That is, we can compute $R-\sum\limits_{a=0}^{q-1}R'(a)$ as:
	 		\[
	 			 \sum\limits_{j=q}^{q+p} (U\squarebin (V\circ[0,\dots,j-q]))\otimes ((U\circ[0])\squarebin (V\circ[j-q,\dots,p])) = (U\wedge (U\circ[0]))\twisty\triangle_B(V)
	 		\]
	 		Adding it all up, we have $\nabla^2([U]\otimes [V])$ is represented by
	 		\begin{gather*}
	 			(U\wedge (U\circ[0]))\twisty\triangle_B(V)+\left(\sum\limits_{a=0}^{q-1}U\circ[0,\dots,a]\wedge U\circ[a,\dots,q]\right)\twisty\triangle_B(V) \\
	 			= \left(\sum\limits_{a=0}^q ((U\circ[0,\dots,a])\wedge (U\circ[a,\dots,q]))\right)\twisty \triangle_B(V) \\
	 			= \triangle_F(U)\twisty\triangle_B(V)
	 		\end{gather*}
	 		But this is exactly the element of $C_{*}(B;C_*(F))\otimes C_*(B;C_*(F))$ that represents the comultiplication of Lemma \ref{new comult} so we are done. 
	 	\end{proof}
 	
 	\section{Example Computations}\label{section: examples}
 		After putting in the work to put a comultiplication in the Serre spectral sequence we now show how it can be used to make some calculations.  The 	comultiplication plays a role in resolving differentials that is completely dual to that of the multiplication in the cohomological Serre spectral sequence.  We make no effort to be efficient in our writing, preferring instead to be more didactic in our examples.
 		\par
 		We begin by computing the homology of $U(2)$, the second unitary group.  In addition to calculating $H_*(U(2))$ as a vector space, we will also identify the $k$-coalgebra structure.  While this calculation can be done other ways, we find it to be a helpful example for several reasons.  First, it gives us an easy first case in using the co-Leibniz rule, which is less well-known that the usual Leibniz rule.  Second, it allows us to highlight certain subtleties that arise when trying to deduce the coalgebra structure on the homology of total spaces of fibrations. 
		
		\begin{example}
			We will make use of the fibration:
			\[
				U(1)\to U(2)\to S^3
			\]
			Since there is a homeomorphism $U(1)\cong S^1$, we know the $k$-homology coalgebras of both the fiber and the base: we have $H_0(S^1) = H_0(S^3) = k$, $H_1(S^1)=k\{\widetilde{x}_1\}$, $H_3(S^3) = k\{\widetilde{x}_3\}$ and both $\widetilde{x}_i$ are primitive.  The $E^2$ page of the associated homological Serre spectral sequence is:
			
			\begin{sseqdata}[ name = U,
				homological Serre grading,
				yscale =.75,
				xscale=1.25,
				y axis gap = 20pt,
				classes = {draw = none}]
				
				\class["1\squarebin 1"](0,0)
				\class["{\widetilde{x}_3\squarebin 1}"](3,0)
				\class["{1\squarebin \widetilde{x}_1}"](0,1)
				\class["{\widetilde{x}_3\squarebin \widetilde{x}_1}"](3,1)

			\end{sseqdata}
			\begin{center}
				\printpage[ name = U, page = 2]
			\end{center} 
			
			For dimension reasons the spectral sequence collapses here, so $E^{\infty}=E^2$ and we can read off the homology groups of $U(2)$: $H_n(U(2))=k$ when $n=0,1,3,4$ and $0$ else.  We will call the generators of the non-trivial homology groups $1,x_1,x_3,$ and $x_{3,1}$ respectively.
			
			We would like to determine the coalgebra structure on $H_*(U(2))$.  For dimension reasons it is clear that $x_1$ and $x_3$ are both primitive, so it remains to compute $\triangle(x_{3,1})  = \triangle_{U(2)}(x_{3,1})$.  For dimension reasons, we must have:
			\[
				\triangle(x_{3,1}) = x_{3,1}\otimes1 + \epsilon_1(x_1\otimes x_3)+\epsilon_2(x_3\otimes x_1) + 1\otimes x_{3,1} 
			\]
			where the $\epsilon_i\in k$.  To decide which, we will use the fact that our spectral sequence converges to $H_*(U(2))$ as a coalgebra in the sense of Definition \ref{convergence as coalgebras}.  That is, we have an isomorphism of bigraded coalgebras $\Gr(H_*(U(2)))\cong E^{\infty}$, where the filtration on the homology $H_*(U(2))$ is induced by the Serre filtration.  Since we know the $E^{\infty}$ page exactly, we can recover the filtration on homology using the isomorphism \ref{filtration quotients}.  We have:
			\begin{align*}
				F^{-1} & =0\\
				F^2 & = F^1 = F^0 =\mathbb{F}_2\{1,x_1\}\\
				F^3 & = \mathbb{F}_2\{1,x_1,x_3,x_{3,1}\} =  H_*(U(2)) 
			\end{align*}
			Since we know the filtration, we can compute the induced comultiplication in $\Gr(H_*(U(2)))$.  We compute
			\begin{equation}\label{comult in gr}
				\Gr(\triangle)([x_{3,1}]) = [x_{3,1}]\otimes[1] + \epsilon_1([x_1]\otimes [x_3])+\epsilon_2([x_3]\otimes [x_1]) + [1]\otimes [x_{3,1}] 
			\end{equation}
			where the square brackets denote elements in the associated graded.  Since the associated graded and $E^{\infty}$ page are isomorphic as coalgebras we can compute the $\epsilon_i$ in equation \ref{comult in gr} using the comultiplication $\nabla^{\infty}$ on $E^{\infty}$.  Since our isomorphism takes $[x_{3,1}]\in \Gr(H_*(U(2)))$ to $\widetilde{x}_3\squarebin \widetilde{x}_1\in E^{\infty}$, we compute $\nabla^{\infty}(\widetilde{x}_3\squarebin \widetilde{x}_1) = \triangle_{S^3}(\widetilde{x_3})\twisty \triangle_{S^1}(\widetilde{x_1})$ to be  
			\[
			 	(\widetilde{x_3}\squarebin\widetilde{x_1})\otimes (1\squarebin 1) - (1\squarebin \widetilde{x_1})\otimes (\widetilde{x_3}\squarebin 1) + (\widetilde{x_3}\squarebin 1)\otimes (1\squarebin \widetilde{x_1}) + (1\squarebin 1)\otimes (\widetilde{x_3}\squarebin\widetilde{x_1})
			\]
			and so $\epsilon_1=-1$ and $\epsilon_2=1$.
		\end{example}
		
		In this example we are lucky there is an isomorphism of \emph{coalgebras} $H_*(U(2))\cong \Gr(H_*(U(2)))$.  In general this is too much to hope for because there are non-isomorphic filtered graded coalgebras $C$ and $D$ whose associated gradeds are isomorphic as bigraded coalgbras.  Nevertheless, we will be able to perform similar computations for the higher dimensional unitary groups. This requires a brief diversion to discuss cofree coalgebras. 
		
		Let $C=C(V)$ be a coaugmented graded coalgebra with coaugmentation $\eta\colon k\to C$.  We will write $\overline{C}$ for the graded vector space $\coker \eta$.  
		\begin{defn}
			Let $V$ be a finite dimensional, non-negatively graded vector space and $C$ a graded, cocommutative coaugmented coalgebra.  We say $C$ is the graded cocommutative cofree coaugmented algebra cogenerated by $V$ if:
			\begin{enumerate}
				\item There is a vector space map $p\colon \overline{C}\to V$.
				\item Given any graded cocommutative coaugmented coalgebra $D$, and graded linear map $f\colon\overline{D}\to V$ there is a unique coalgebra map $g\colon D\to C$ so that  $f$ is equal to the composition 
				\[
					\overline{D}\xrightarrow{\overline{g}}\overline{C}\xrightarrow{p} V
				\] 
			\end{enumerate}
			We will also refer to such coalgebras as simply cofreely cogenerated by $V$.
		\end{defn}
	
		\begin{remark}
			Our definition for a cofree coalgebra could be expressed more categorically as follows.  Let $\textrm{CoAlg}_k$ be the category of graded cocommutative coaugmented $k$-coalgebras and let $\textrm{Vec}_k$ be the category of graded $k$ vector spaces.  The construction $C\mapsto \overline{C}$ is a functor $U\colon\textrm{CoAlg}_k\to \textrm{Vec}_k$ and it is not hard to check that a coalgebra $C(V)$ is cofreely cogenerated by a graded vector space $V$ if and only if there is a natural bijection of sets:
			\[
				\textrm{CoAlg}_k(D,C(V))\cong \textrm{Vec}_k(\overline{D},V)
			\]
			
			That is, constructing a $C(V)$ for all $V$ is same as building a right adjoint to the functor $U$.  A construction of such a functor, as well as a description of all the examples that follow, can be found in \cite[Section 3]{BGHSZ18}.  For a more general discussion of cofree coalgebras see \cite{Sweed}.
		\end{remark}
	
		\begin{example}
			Let $V$ be the graded vector space that is $0$ is every dimension.  Then the cofree coalgebra generated by $V$ is simply a copy of $k$ in degree $0$ with comultiplication $\triangle\colon k\to k\otimes k =k$ given by the identity map.  
		\end{example}
		\begin{example}\label{cofree on one generator}
			Suppose $\textrm{char}( k)\neq 2$ and let $V(n)$ be a copy of $k$ in dimension $n$ odd.  Then the coalgebra cofreely generated by $V(n)$ has a copy of $k$ in dimensions $0$ and $n$, with generator $1$ and $x_n$.  The comultiplication is specified by $\triangle(1)=1\otimes 1$ and $x_n$ is primitive.  The map $p\colon \overline{C}\to V$ sends $x_n$ to $1\in V(n)_n$.  As a graded vector space, we note that $C(V(n))$ is isomorphic to the exterior algebra $\Lambda_k[x_n]$.
		\end{example}
		\begin{example}\label{cofree on odd generators}
			Building on the last example, suppose $\textrm{char}( k)\neq 2$ and let $V$ have copies of $k$ in dimensions $1,3,\dots, 2n-1$ for some $n$.  Clearly $V$ is isomorphic to the product $V(1)\times V(3)\times\dots V(2n-1)$.  Since the cofree coalgebra functor is a right adjoint it preserves products and so we have
			\[
				C(V)\cong \Lambda_k[x_1]\otimes\Lambda_k[x_3]\otimes \dots\otimes\Lambda_k[x_{2n-1}] \cong \Lambda_k[x_1,\dots, x_{2n-1}] 
			\]
			Here we have used the fact that the categorical product of graded cocommutative coalgebras is the tensor product \cite[Proposition 20.3.4]{mayponto}.
		\end{example}
		
		We now turn to the computation of homology of the unitary groups $U(n)$ for $n>2$.  We will work over a field $k$ with $\textrm{char}(k)\neq 2$.  The idea is to work inductively up the family of fibrations
		\begin{equation}\label{Unitary Fibrations}
			U(n-1)\to U(n)\to S^{2n-1}
		\end{equation}
		
		First, we need a lemma.
		\begin{lem}\label{primitive_element_lemma}
			Let $(E,d,\triangle)$ be a spectral sequences of coalgebras, and let $x\in E^r$ be primitive, i.e. $\triangle^r(x) = x\otimes 1+1\otimes x$.  Then $d^r(x)$ is also primitive.
		\end{lem}
		\begin{proof}
			Using the co-Leibniz rule we get
			\begin{align*}
			\triangle^r(d^r(x))  = (d^r\otimes 1+ 1\otimes d^r)(\triangle^r(x)) = d^r(x)\otimes 1+1\times d^r(x) 
			\end{align*} 
		\end{proof}
	
		\begin{proposition}
			Let $k$ be a field with $\textrm{char}(k)\neq 2$ and let $V_n$ be the graded vector space that is a copy of $k$ in dimensions $1,3,\dots, 2n-1$.  Then $H_*(U(n);k)$ is the coalgebra cofreely generated by $V_n$ as described in Example \ref{cofree on odd generators}.
		\end{proposition}
		\begin{proof}
			We work by induction on $n$, the base case $n=1$ follows from Example \ref{cofree on one generator} and the fact that $U(1)\cong S^1$.  So suppose we have proven the result for $U(n-1)$ and consider the homological Serre spectral sequence associated to the fibration \eqref{Unitary Fibrations}.
			
			The $E^2$ page, as a vector space, is given by 
			\[
				H_*(S^{2n-1})\otimes H_*(U(n-1)) \cong \Lambda_k[x_{2n-1}]\otimes C(V_{n-1})\cong C(V_n)
			\]
			To make the identification of $H_*(U(n);k)$ as a vector space it remains to show that the spectral sequence collapses at the $E^2$ page.  That is, we must show that all differentials in the spectral sequence are always zero. Because the homology of the base is trivial except in dimensions $0$ and $2n-1$ we have $E^2=E^3=\dots=E^{2n-1}$ and the only differential that could be non-trivial is $d=d^{2n-1}$. The $E^2=E^7$ page of the spectral sequence for $n=4$ is displayed below.
			\begin{sseqdata}[ name = more U,
				homological Serre grading,
				yscale =.75 , 
				classes = {draw = none} ]
				\class["1\squarebin 1"](0,0)
				\class["\mathbf{1\squarebin x_{1}}"](0,1)
				\class["\mathbf{1\squarebin x_{3}}"](0,3)
				\class["1\squarebin x_{3}x_1"](0,4)
				\class["\mathbf{1\squarebin x_{5}}"](0,5)
				\class["1\squarebin x_{5}x_1"](0,6)
				\class["1\squarebin x_{5}x_3"](0,8)
				\class["1\squarebin x_{5}x_3x_1"](0,9)
				\class["\mathbf{x_{7}\squarebin 1}"](7,0)
				\class["x_{7}\squarebin x_{1}"](7,1)
				\class["x_{7}\squarebin x_{3}"](7,3)
				\class["x_{7}\squarebin x_{3}x_1"](7,4)
				\class["x_{7}\squarebin x_{5}"](7,5)
				\class["x_{7}\squarebin x_{5}x_1"](7,6)
				\class["x_{7}\squarebin x_{5}x_3"](7,8)
				\class["x_{7}\squarebin x_{5}x_3x_1"](7,9)
				\class(-1,0)
				\class(8,0)
			\end{sseqdata}
			\begin{center}
				\begin{figure}[h]
					\printpage[ name = more U, page = 2 ]
					\caption{The $E^2=E^7$ page of Serre spectral sequence for the fibration  $U(4)\to U(5)\to S^7$.  The primitive elements are displayed in bold.}
				\end{figure}
			\end{center}
			
			By induction, the only non zero homogeneous primitive elements in $H_*(U(n-1))$ are $x_{1},\dots,x_{2n-3}$.  Lemma \ref{primitive_element_lemma} tells us that $d(\widetilde{x}_{2n-1}\squarebin 1)$ is a primitive element in $E^{2n-1}_{0,2n-2}$.  But this impossible unless this element is $0$ because $E^{2n-1}_{0,*}\cong H_*(U(n-1))$ which only has primitive elements in odd dimensions.  Thus $d(x_{2n-1}\squarebin 1)=0$.
			
			Every other generator of the $2n-1$ column of the spectral sequence is of the form $x_{2n-1}\squarebin y$ for $y\in H_*(U(n-1))$.  Fix a particular $y$, and suppose that we know that $d(x_{2n-1}\squarebin z)=0$ for all $|z|<|y|$.  A routine computation with the co-Leibniz rule tells us that $d(x_{2n-1}\squarebin y)$ is a primitive element.  But $|d(x_{2n-1}\squarebin y)|=2n+|y|-1>2n-3$ and the highest degree of a non-zero primitive element in $H_*(U(n-1))$ is $2n-3$ so  we conclude that $d(x_{2n-1}\squarebin y)=0$. 
			
			It remains to check that the proposed comultiplication is correct.  That is, we need to show that the apparent vector space isomorphism $H=H_*(U(n);k)\cong C(V_n)$ is an isomorphism of graded coalgebras. First, we filter $C(V_n)$ by giving each element the degree equal to the $p$ coordinate of the corresponding element in $E^{\infty}$.  That is, the filtration degree of $x_{i_1}x_{i_2}\dots x_{i_r}$ is $2n-1$ if $i_k=2n-1$ for some $k$ and $0$ else. The degree of a sum of elements is the largest degree appearing in the sum.  We note that the associated graded of $C(V_n)$ with this filtration is exactly $E^{\infty}$.
			
			Now, define a $k$-linear map $\overline{H}\to V_n$ by sending $x_{i}$ to the generator of $V_n$ in degree $i$ and all other terms to $0$.  By the universal property of cofreeness, this extends to a map of filtered coalgebras $H\to C(V_n)$.  On associated bigraded coalgebras this map is the identity on the $E^{\infty}$ page and so the result follows from the fact that a morphism of filtered coalgebras is an an isomorphism if it is an isomorphism on associated bigradeds.
			\end{proof}

	\bibliography{math-comult}	
	\bibliographystyle{alpha}
	
\end{document}